\batchmode
\makeatletter
\def\input@path{{"/Users/russw/Documents/Research/mypapers/Invariable generation of alternating groups/"}}
\makeatother
\documentclass[12pt,oneside,english]{amsart}
\usepackage{lmodern}
\usepackage[T1]{fontenc}
\usepackage[latin9]{inputenc}
\usepackage{geometry}
\usepackage{babel}
\usepackage{booktabs}
\usepackage{url}
\usepackage{enumitem}
\usepackage{amsthm}
\usepackage{amssymb}
\usepackage[pdftex,unicode=true,pdfusetitle,
 bookmarks=true,bookmarksnumbered=false,bookmarksopen=false,
 breaklinks=false,pdfborder={0 0 1},backref=false,colorlinks=false]
 {hyperref}

\makeatletter

\providecommand{\tabularnewline}{\\}

\numberwithin{equation}{section}
\numberwithin{figure}{section}
\theoremstyle{plain}
\newtheorem{thm}{\protect\theoremname}[section]
\theoremstyle{plain}
\newtheorem{question}[thm]{\protect\questionname}
\theoremstyle{plain}
\newtheorem{cor}[thm]{\protect\corollaryname}
\theoremstyle{remark}
\newtheorem{rem}[thm]{\protect\remarkname}
\theoremstyle{plain}
\newtheorem{lem}[thm]{\protect\lemmaname}
\theoremstyle{remark}
\newtheorem{notation}[thm]{\protect\notationname}
\theoremstyle{definition}
\newtheorem{example}[thm]{\protect\examplename}
\theoremstyle{remark}
\newtheorem{note}[thm]{\protect\notename}

\pdfoutput=1  

\makeatother

\providecommand{\corollaryname}{Corollary}
\providecommand{\examplename}{Example}
\providecommand{\lemmaname}{Lemma}
\providecommand{\notationname}{Notation}
\providecommand{\notename}{Note}
\providecommand{\questionname}{Question}
\providecommand{\remarkname}{Remark}
\providecommand{\theoremname}{Theorem}

\begin{document}
\global\long\def\comma{{,}}%

\title[On invariable generation of alternating groups]{On invariable generation of alternating groups by elements of prime
and prime power order}
\author{Robert M. Guralnick, John Shareshian, and Russ Woodroofe}
\address{Department of Mathematics, University of Southern California, 3620
S. Vermont Ave Los Angeles, CA 90089-2532 USA}
\email{guralnic@usc.edu}
\address{Department of Mathematics, Washington University in St.~Louis, St.~Louis,
MO, 63130}
\email{jshareshian@wustl.edu}
\address{Univerza na Primorskem, Glagoljaška 8, 6000 Koper, Slovenia}
\email{russ.woodroofe@famnit.upr.si}
\urladdr{\url{https://osebje.famnit.upr.si/~russ.woodroofe/} }
\thanks{Work of the first author was partially supported by NSF Grant DMS-1901595
and Simons Foundation Fellowship 609771. Work of the second author
was supported in part by NSF Grant DMS 1518389. Work of the third
author is supported in part by the Slovenian Research Agency research
program P1-0285 and research projects J1-9108, N1-0160, J1-2451, J3-3003.\medskip{}
}
\thanks{Github: \texttt{\href{https://github.com/RussWoodroofe/invgen}{RussWoodroofe/invgen}}$\quad\quad$Dataset
doi:\href{https://doi.org/10.5281/zenodo.5914767}{10.5281/zenodo.5914767}}
\begin{abstract}
We verify that every alternating group of degree at most one quadrillion
is invariably generated by an element of prime order together with
an element of prime power order.
\end{abstract}

\maketitle

\section{\label{sec:Introduction}Introduction}

We say that a finite group $G$ is \emph{invariably generated} by
the elements $h,g\in G$ if for every $x,y\in G$ it holds that $\left\langle h^{x},g^{y}\right\rangle =G$.
This paper is motivated by the following question, asked by the first
author in work with Dolfi, Herzog, and Praeger \cite{Dolfi/Guralnick/Herzog/Praeger:2012}.
\begin{question}
\label{que:MainQuestion}Which finite simple groups are invariably
generated by two elements of prime order? Of prime-power order?
\end{question}

There are applications of invariable generation to computational Galois
theory \cite{Dixon:1992,Eberhard/Ford/Green:2017}, where it can be
used to verify that the Galois group of a given polynomial is the
full symmetric group. Invariable generation also yields useful fixed-point
free actions on certain simplicial complexes that may be derived from
a group \cite{Shareshian/Woodroofe:2016}.

Problems related to Question~\ref{que:MainQuestion} have been studied
elsewhere. King \cite{King:2017} shows that every finite simple group
is generated (not necessarily invariably) by two elements of prime
order. Every finite simple group is invariably generated by two elements
(of unspecified order) \cite{Guralnick/Malle:2012,Kantor/Lubotzky/Shalev:2011}.
Invariable generation by a small number of random elements \cite{Dixon:1992,Eberhard/Ford/Green:2017,Pemantle/Peres/Rivin:2016,McKemmie:2021}
or by a sequence of random elements \cite{Kantor/Lubotzky/Shalev:2011,Lucchini:2018}
has received a fair bit of recent attention. Other related work on
invariable generation includes \cite{Detomi/Lucchini:2015b,Detomi/Lucchini:2015,Garzoni/Lucchini:2020}.

It seems reasonable to believe, and was conjectured in \cite{Dolfi/Guralnick/Herzog/Praeger:2012},
that all but finitely many simple groups of Lie type are invariably
generated by two elements of prime order. For alternating groups,
this is not the case. The second two authors gave infinitely many
counterexamples in \cite{Shareshian/Woodroofe:2018}, by showing that
$A_{n}$ is not invariably generated by elements of prime order whenever
$n>4$ is a power of $2$. On the other hand, another result of \cite{Shareshian/Woodroofe:2018}
shows that the set of integers $n$ where $A_{n}$ is so generated
has asymptotic density of $1$ under assumption of the Riemann Hypothesis
(and very close to 1 with no assumption). 

In the current paper, we attack the problem of invariable generation
of alternating groups by prime-power elements computationally, improving
significantly on the prior computational results in \cite{Shareshian/Woodroofe:2018}.
Our main result is as follows.
\begin{thm}
\label{thm:MainTheorem}For all $n$ between $5$ and $10^{15}$,
the alternating group $A_{n}$ is invariably generated by an element
of prime power order $p^{a}$ together with an element of prime order
$r$.

Moreover, the prime $r$ can be chosen for every $n>24$ in this range
so that $r\geq2\sqrt{n}$, while $p^{a}$ can be chosen to be one
of the three largest prime powers dividing $n$ for every $n$ in
this range except for $199\comma445\comma521\comma968$ and $6\comma421\comma990\comma708\comma848$.
\end{thm}

The computation, as we will describe in greater detail below, uses
an extension of the segmented sieve of Eratosthenes. A key observation
is to notice that while we are sieving, we can simultaneously identify
the integers that fail to be power-smooth, while also identifying
the largest prime factor of each power-smooth integer.

After initially prototyping code in GAP \cite{GAP4.12}, we implemented
the most speed-critical part in C, using the primesieve library \cite{Walisch:2022}.
We ran this code using 15 concurrent processes for about 225 hours
(or about 3062 process-hours) on a 16 core 3.3 GHz Intel Xeon E5 server,
resulting in a list of numbers failing Lemma~\ref{lem:GoodSieve}
below. We then checked this list of numbers in GAP on a 2.4 GHz Intel
Core i5 MacBook Pro from 2019, which took about another 14 hours (single-threaded),
and that completed the verification of the theorem.

The condition of invariable generation by an element of prime-power
order and an element of prime order may be relaxed in several natural
ways, as is studied in detail in \cite{Shareshian/Woodroofe:2018};
see also \cite{Casacuberta:2020}. Is every alternating group invariably
generated by two elements of prime-power order? Or by any Sylow subgroups
taken at two suitably-selected primes? All of these problems appear
to remain open (beyond $10^{15}$).

The question of generation by any Sylow subgroups taken at two primes
is shown in \cite[Theorem 1.3]{Shareshian/Woodroofe:2018} to be equivalent
to an elementary question on divisibility of binomial coefficients
by the same primes. Working across this equivalence, the following
is a straightforward corollary of Theorem~\ref{thm:MainTheorem}.
\begin{cor}
For all $n$ between $5$ and $10^{15}$, there are primes $p$ and
$r$ so that every nontrivial binomial coefficient ${n \choose k}$
is divisible by at least one of $p,r$.
\end{cor}

The results of this paper improve on the computational results in
\cite{Shareshian/Woodroofe:2018} in two main ways: First, we examine
a more restrictive generation condition. Second, we make better use
of known algorithms and approaches from the computational number theory
literature.

We expect to comprehensively address Question~\ref{que:MainQuestion}
for sporadic simple groups and for simple groups of Lie type in a
forthcoming paper. 

This paper is organized as follows. In Section~\ref{sec:Background},
we give the necessary background from number theory and from the theory
of alternating and symmetric groups. In Section~\ref{sec:Lemmas},
we give known and new lemmas used in the computation. In Section~\ref{sec:Strategy},
we present the details of our computation. In Section~\ref{sec:Discussion},
we discuss relationships with the number theory literature and possible
directions for future research.

\section*{Acknowledgements}

We thank Steven Xiao at Washington University in St.~Louis for giving
us access to the computing hardware for our main computation, and
Jernej Vi\v{c}i\v{c} at the University of Primorska for providing
computer hardware used in earlier versions of the computation. Chris
Jefferson and Marcus Pfeiffer answered some questions about GAP memory
management. Kaisa Matomäki explained several aspects of her work and
that of Joni Teräväinen (see Section~\ref{sec:Discussion}), and
of the number theory literature more broadly.

\section{\label{sec:Background}Background and notation}

Our computation will combine number theory and group theory. A positive
integer is \emph{$B$-smooth} if it has no prime factor greater than
$B$, and \emph{$B$-power-smooth} if it has no prime-power factor
greater than $B$.

In order for a set of elements $S$ to generate a group $G$, it is
necessary and sufficient that no maximal subgroup of $G$ contains
$S$. The maximal subgroups of $A_{n}$ are well-known, and come in
three flavors. 
\begin{itemize}
\item The \emph{intransitive maximal subgroups}, which fix a nonempty proper
subset of $\left[n\right]$ under the natural action. Each such subgroup
is isomorphic to a subgroup of index $2$ in $S_{k}\times S_{n-k}$
for some $1\leq k\leq n-1$, and has index ${n \choose k}$ in $A_{n}$.
\item The \emph{imprimitive maximal subgroups}, which are transitive but
fix a partition of $\left[n\right]$ under the natural action. Each
such subgroup is isomorphic to a subgroup of index $2$ in $S_{d}\wr S_{n/d}$
for some proper divisor $d$ of $n$, and has index ${n \choose d,d,\dots,d}/(n/d)!$
in $A_{n}$.
\item The \emph{primitive maximal subgroups}, which satisfy neither of the
above. This is the most difficult class to handle, but may be approached
with the Classification of Finite Simple Groups and other results.
\end{itemize}
We discuss the primitive subgroups first. A classic theorem of Jordan
is already a useful tool.
\begin{thm}[Jordan \cite{Jordan:1875}]
\label{thm:JordanPrim} No proper primitive subgroup of $A_{n}$
contains a $p$-cycle for any prime $p$ with $p\leq n-3$.
\end{thm}

One approach to proving Theorem~\ref{thm:JordanPrim} is to note
that a primitive subgroup of $S_{n}$ containing an $n-k$ cycle is
$(k+1)$-transitive (see e.g. \cite[Exercise 7.4.11]{Dixon/Mortimer:1996}),
a strong restriction. Combining this idea with the Classification
yields substantial improvements to Theorem~\ref{thm:JordanPrim}.
We will make use of the following such improvement.
\begin{thm}
\label{thm:JonesPrim} If a proper primitive subgroup $H$ of $A_{n}$
contains an $(n-2)$-cycle, then $n-1$ is a prime power.
\end{thm}

\begin{proof}
As we have observed, the action of $H$ is $3$-transitive. The $3$-transitive
subgroups of $A_{n}$ are known from the Classification, as is laid
out accessibly in \cite[table following Theorem 5.3]{Cameron:1981}.
In particular, if $n>24$ and $H\neq A_{n}$, then $H$ has a minimal
normal subgroup that is isomorphic to $PSL(2,q)$, and $n=q+1$.
\end{proof}
\begin{rem}
A further generalization of Theorem~\ref{thm:JordanPrim} and Theorem~\ref{thm:JonesPrim}
may be found in \cite{Jones:2014}, where the Classification is used
to completely characterize the primitive subgroups of $A_{n}$ that
contain a cycle of any length.
\end{rem}

Another line of generalization of Theorem~\ref{thm:JordanPrim} is
to replace the $p$-cycle with a product of a small number of $p$-cycles.
Praeger did significant work along these lines in \cite{Praeger:1975,Praeger:1976,Praeger:1979},
shortly before the completion of the Classification Theorem. Later,
Liebeck and Saxl in \cite{Liebeck/Saxl:1985b} used the Classification
to describe completely (for large enough $r$) the groups containing
a product of fewer than $r$ cycles of length $r$. Although a full
statement of their results takes several pages, we use the following
consequence.
\begin{thm}[{Liebeck and Saxl \cite[Tables 1 and 2]{Liebeck/Saxl:1985b}}]
\label{thm:LiebeckSaxlPrim} Let $r$ be a prime with $\sqrt{n}<r$,
and let $G$ be a proper primitive subgroup of $A_{n}$. If $y\in G$
is the product of $\lfloor\frac{n}{r}\rfloor$ disjoint $r$-cycles
with $k=n-\lfloor\frac{n}{r}\rfloor\cdot r$ fixed points, then one
of the following holds:
\begin{enumerate}
\item $n={c \choose 2}$, where $2+r\leq c\leq\frac{3}{2}r-\frac{1}{2}$,
and $k=\frac{1}{2}(r^{2}+r-2rc-c+c^{2})$.
\item $n=\frac{q^{d}-1}{q-1}$, where $r$ divides $q^{i}-1$, and $k=q^{d-i-1}+\cdots+1=\frac{q^{d-i}-1}{q-1}$.
\item $k\leq2$.
\item $n=24$.
\item $n$ is a prime power.
\end{enumerate}
\end{thm}

\begin{rem}
In the case of Theorem~\ref{thm:LiebeckSaxlPrim} where $n={c \choose 2}$,
then 
\[
{c \choose 2}=n\geq\frac{1}{2}(r+2)(r+1),
\]
so $2n\geq(r+1)^{2}$. In particular, $\sqrt{2n}-1\geq r$.
\end{rem}

After handling the primitive subgroups with Theorem~\ref{thm:LiebeckSaxlPrim},
the intransitive and imprimitive subgroups are dealt with using number
theory (after some additional work). We will frequently use the following
fact, which may be found for example in \cite[after Theorem 7.27]{Rotman:1995}.
\begin{lem}
\label{lem:SylowSgOrbits}If $P$ is a Sylow $p$-subgroup of $A_{n}$,
then the orbits of $P$ in the natural action of $A_{n}$ on $\left[n\right]$
agree with the base $p$ representation of $n$. 

More precisely, if $n=\alpha_{0}p^{0}+\alpha_{1}p^{1}+\cdots$, then
$P$ has $\alpha_{0}$ fixed points, $\alpha_{1}$ orbits of order
$p$, $\alpha_{2}$ orbits of order $p^{2}$, and so forth.
\end{lem}

\section{\label{sec:Lemmas}Lemmas}

In this section, we give several lemmas that will be useful in the
computation. First, it is easy to handle powers of primes.
\begin{lem}
\cite[Proposition 1.4B and proof]{Shareshian/Woodroofe:2018}\label{lem:PrimePower}
If $n$ is a power of the prime $p$, then $A_{n}$ is invariably
generated by an $r$-cycle for any prime $r$ with $n/2<r<n-2$, together
with an element of $p$-power order.
\end{lem}

We now fix the following notation.
\begin{notation}
Throughout the following, $n$ will be a large integer that is not
a prime power, $p$ will be a prime and $p^{a}$ will be a prime-power
divisor of $n$, and $r$ will be another prime that is smaller than
$n$ (but which is ``fairly large'').
\end{notation}

The following is very similar to results of \cite[Lemma 1.8]{Shareshian/Woodroofe:2018},
and is proved in exactly the same manner.
\begin{lem}
\label{lem:GoodSieve}If $r$ and $p^{a}$ are such that $r<n-2<n<r+p^{a}$,
then $A_{n}$ is invariably generated by an $r$-cycle together with
any fixed-point-free permutation $x$ having no cycle of length less
than $p^{a}$.

Similarly if $r=n-2$ and $n-1$ is not a power of $2$. 
\end{lem}

\begin{proof}
An $r$-cycle $y$ avoids all primitive proper subgroups by Theorems
\ref{thm:JordanPrim} and \ref{thm:JonesPrim}. Moreover, the orders
of the maximal imprimitive subgroups are not divisible by $r$. This
leaves the intransitive maximal subgroups. Since $r+p^{a}>n$, the
support of every cycle of $x$ intersects the support of the unique
nontrivial cycle of $y$, hence the subgroup generated by $x$ and
$y$ is transitive.
\end{proof}
We say that $g\in A_{n}$ or $S_{n}$ is a \emph{base-$p$ element
}if the cycle structure of $g$ agrees with the base $p$ representation
of $n$. That is, if $n$ has base $p$ representation $\alpha_{0}p^{0}+\alpha_{1}p^{1}+\cdots$,
then a base-$p$ element of $A_{n}$ has $\alpha_{0}$ fixed points,
$\alpha_{1}$ cycles of length $p$, $\alpha_{2}$ cycles of length
$p^{2}$, and so forth. Thus, by Lemma~\ref{lem:SylowSgOrbits},
the orbits of a base-$p$ element coincide with those of a Sylow $p$-subgroup.
It is clear that $A_{n}$ has a base-$p$ element for every $p\neq2$,
and that $A_{n}$ has a base-$2$ element if and only if the digit
sum $\alpha_{1}+\alpha_{2}+\cdots$ is even.

We can take the element $x$ in Lemma~\ref{lem:GoodSieve} to be
a base-$p$ element when one exists. More broadly, it is not difficult
to describe the intransitive and imprimitive maximal subgroups containing
a base-$p$ element.
\begin{lem}
\label{lem:BasepDivis}A base-$p$ element of $A_{n}$ fixes some
set of size $i$ in the natural action if and only if $p\nmid{n \choose i}$.
\end{lem}

\begin{proof}
The orbit system in the action on $\left[n\right]$ of a Sylow subgroup
and of a base-$p$ element coincide. The result now follows from the
Orbit-Stabilizer theorem.
\end{proof}
There is also a simple condition that suffices to show that a base-$p$
element does not fix a block system, as follows.
\begin{lem}
\label{lem:CarryConditionImprim}Suppose that the base-$p$ element
$x$ of $A_{n}$ preserves a partition $\pi$ having $e$ blocks of
size $d$. Then no carry occurs when multiplying $d$ by $e$ in base
$p$. Equivalently, if $d=\sum_{i\geq0}\delta_{i}p^{i}$ and $e=\sum_{i\geq0}\epsilon_{i}p^{i}$
are the base-$p$ representations for $d,e$, then for each $k$ it
holds that $\sum_{0\leq i\leq k}\delta_{i}\epsilon_{k-i}<p$.
\end{lem}

\begin{proof}
The equivalence of the two conclusions is straightforward from definitions.

Let $H$ be the stabilizer in $S_{n}$ of $\pi$. By basic Sylow theory,
the subgroup $\left\langle x\right\rangle $ is contained in a Sylow
$p$-subgroup $P$ of $H$, which is contained in a Sylow $p$-subgroup
of $S_{n}$. It follows from Lemma~\ref{lem:SylowSgOrbits} that
the orbits of the three subgroups agree.

An order argument shows that $P\cong P_{d}\wr P_{e}$, where $P_{d}$
and $P_{e}$ are respectively Sylow $p$-subgroup of $S_{d}$ and
$S_{e}$. (Here, $S_{d}$ acts on a block of $\pi$, while $S_{e}$
acts on the set of blocks.) It now follows directly (see for example
\cite[Theorem 7.25]{Rotman:1995}) that for each $k\geq0$, the action
of $P$ on $[n]$ has $\alpha_{k}=\sum_{0\leq i\leq k}\delta_{i}\epsilon_{k-i}$
orbits of size $p^{k}$. By definition of the base-$p$ representation
of $n$, we must have $\alpha_{k}<p$.

This completes the proof for odd $p$. For even $p$, we complete
the proof by intersecting with $A_{n}$.
\end{proof}
Lemma~\ref{lem:CarryConditionImprim} is less useful for large primes.
For these we need another condition.
\begin{lem}
\label{lem:DivisImprim}Suppose that $r>\sqrt{n}$, and let the base-$r$
representation of $n$ be $\beta_{0}+\beta_{1}r$, where $\beta_{0}>0$.
If a base-$r$ element $x$ of $A_{n}$ preserves a partition $\pi$
having $e$ blocks of size $d$, then either $d$ or $e$ divides
$\gcd(\beta_{0},\beta_{1})$.
\end{lem}

\begin{proof}
Let $d$ and $e$ respectively have base-$r$ representation $\delta_{0}+\delta_{1}r$
and $\epsilon_{0}+\epsilon_{1}r$. By a similar argument as in the
proof of Lemma~\ref{lem:CarryConditionImprim}, we have $\beta_{0}=\delta_{0}\epsilon_{0}$,
$\beta_{1}=\delta_{0}\epsilon_{1}+\delta_{1}\epsilon_{0}$, and at
least one of $\delta_{1},\epsilon_{1}$ must be $0$.

Now if $\delta_{1}=0$, then $d=\delta_{0}$ divides $\beta_{0}=\delta_{0}\epsilon_{0}$
and $\beta_{1}=\delta_{0}\epsilon_{1}$; while if $\epsilon_{1}=0$,
then $e=\epsilon_{0}$ similarly divides both $\beta_{0}$ and $\beta_{1}$.
\end{proof}

\section{\label{sec:Strategy}Strategy and details of computation}

There are two requirements for verifying Theorem~\ref{thm:MainTheorem}.
We must be able to show that a given $A_{n}$ is generated by an $r$-element
and a $p$-power element, and we must compute very quickly.

\subsection{\label{subsec:StrategyGeneraton}Strategy for checking generation}

The second two authors showed in \cite[Section 5]{Shareshian/Woodroofe:2018}
that (a stronger statement than) the condition of Lemma~\ref{lem:GoodSieve}
holds with high asymptotic density. Indeed, running our code on various
ranges suggests that about $O(\sqrt{m})$ numbers in the range from
$1$ to $m$ fail this condition.

Thus, the strategy for computation will be to find for each $n$ the
largest prime-power factor $p^{a}$ of $n$ and the largest prime
$r$ smaller than (or possibly equal to) $n-2$. If $p^{a}+r>n$,
then a fixed-point free element of order $p^{a}$ and an $r$-cycle
invariably generate. This verifies invariable generation for most
values of $n$.

For the remaining values of $n$, we seek numbers of the form $cr$
strictly between $n-p^{a}$ and $n-2$, where $r$ is a prime and
$c$ is smaller than $\sqrt{n/2}$. The product $g$ of $c$ cycles
of length $r$ is frequently helpful for invariable generation. By
Theorem~\ref{thm:LiebeckSaxlPrim} and the remark following it, this
$g$ is not in any proper primitive subgroup of $A_{n}$, except possibly
if $n=\frac{q^{d}-1}{q-1}$ and $n-c\cdot r=\frac{q^{i}-1}{q-1}$
for some $i\geq2$. In particular, in this case $n-1$ and $n-c\cdot r-1$
have a common prime power divisor, which is easy to detect computationally.

Having eliminated the possibility of a proper primitive subgroup containing
$g$, we must also avoid intransitive and imprimitive subgroups. We
let $p_{1}$ and $p_{2}$ be the primes associated with the largest
two odd prime-power divisors of $n$, and combine $g$ with a base-$p_{1}$
or base-$p_{2}$ element. (We avoid using $2$ here for convenience,
so that we can avoid checking whether $A_{n}$ has a base-2 element.)
We consider two primes, as there are situations where a given prime
may fail to yield transitivity with one or all possible choices of
$r$.
\begin{example}
\label{exa:31416}Consider $n=31416=2^{3}\cdot3\cdot7\cdot11\cdot17$.
There is no prime $r<n$ such that $r+17>n$, so we must look at $r$
such that $c\cdot r$ is close to $n-3$. The prime $r=7853$ is a
useful such choice. However, both the product of four $r$-cycles
and an appropriate base-$17$ element respect a partition of $n$
into $2\cdot7854+2\cdot7854$. Here, the base 17 representation of
$n$ is $12\cdot17+6\cdot17^{2}+6\cdot17^{3}$, and $2\cdot7854=2\cdot7853+2\cdot1=6\cdot17+3\cdot17^{2}+3\cdot17^{3}$.
The base 11 representation of $n$ is the less symmetric $n=7\cdot11+6\cdot11^{2}+1\cdot11^{3}+2\cdot11^{4}$,
and indeed a base-11 element and a product of $7853$-cycles generate
$A_{31416}$.
\end{example}

We check the condition of Lemma~\ref{lem:CarryConditionImprim} on
each $d$ dividing $\gcd(c,n-c\cdot r)$ (using Lemma~\ref{lem:DivisImprim}),
then check Lemma~\ref{lem:BasepDivis} on every $i$ of the form
$a\cdot r+b$. Out to $10^{15}$, there is always such an $r$ satisfying
the desired conditions together with some prime divisor of $n$. 

\subsection{Implementation}

We perform the computation in two phases.

\subsubsection*{Phase 1:}

Because most values of $n$ satisfy the condition from Lemma~\ref{lem:GoodSieve},
it is critical to check this condition quickly. Indeed, almost all
of the time spent in computing is spent in verifying this condition.

We must find all or most of the primes out to the upper end of the
range we are checking, and also find a large prime-power factor of
each integer in the range. It is well-known that the Sieve of Eratosthenes
is a fast algorithm for generating primes, but is space-hungry, taking
$O(m)$ memory. More space-efficient is the Segmented Sieve of Eratosthenes.
The Segmented Sieve uses the fact that each non-prime number $n$
is divisible by some number smaller than $\sqrt{n}$. Thus, to compute
primes out to $m$, we can find all primes smaller than $\sqrt{m}$,
then repeatedly check numbers in segments of size $\sqrt{m}$ for
divisibility by some small prime. This allows us to list primes out
to $m$ using $O(\sqrt{m})$ memory. 

We modify the Segmented Sieve of Eratosthenes algorithm slightly to
iterate over prime powers smaller than $\sqrt{m}$, rather than just
primes. Then, instead of simply marking each $n$ as not prime, we
record the largest-seen prime-power divisor. Thus, at the end of sieving,
we have found the largest prime-power divisor of each $n$ that is
smaller than $\sqrt{m}$. By performing a bit more bookkeeping, we
are able to identify the numbers that are not $B$-power-smooth, for
some $B$ larger than the largest expected prime gap. We record $B$
in place of the largest prime-power divisor of the non-$B$-power-smooth
numbers. See \cite[Chapter 3.2]{Crandall/Pomerance:2005} for similar
variations on the Segmented Sieve of Eratosthenes.

For further practical efficiency on real computer hardware, we use
the primesieve library \cite{Walisch:2022}. This library uses segments
of length less than $\sqrt{m}$, which has advantages for cache efficiency,
but which requires keeping different lists of primes. The library
handles the adjustments for smaller segments transparently, and has
other optimizations. We sieve separately for power-smooth numbers
and large prime-power divisors. 

The upshot is that for each segment, we can quickly calculate the
primes in the segment, and the largest prime-power divisors for power-smooth
numbers in the segment. Now for each number $n$ in each segment,
we add the prime preceding $n-2$ with the largest prime-power divisor
(or a large enough placeholder value if $n$ is not power-smooth).
If this is smaller than $n$, we see if $n-2$ is prime and $n-1$
is not a power of $2$. If this fails, we record the number (together
with its largest prime-power divisor) for Phase~2. 
\begin{rem}
It is worth remarking that, for $B$ large enough, we do not expect
to ever see a number that passes to Phase~2 which is not $B$-power-smooth.
That is, we expect to see primes on any interval of sufficient length
before $m$. Our code uses $B=5(\log_{2}m)^{2}$, and gives a warning
if it ever does pass a non-$B$-power-smooth integer on to Phase~2.
\end{rem}

The output from Phase~1 in our computation to $10^{15}$ is publicly
available as a dataset on the Zenodo repository \cite{Guralnick/Shareshian/Woodroofe:2022a},
and comprises about $26.5$ million numbers. The first segment (comprising
the 0.8 million fairly power-smooth numbers out to 1 trillion) is
also posted as an arXiv ancillary file. Code for both phases of the
computation is available on GitHub \cite{Guralnick/Shareshian/Woodroofe:2022b}.

\subsubsection*{Phase 2:}

At the conclusion of Phase 1, we have a relatively small number of
fairly power-smooth numbers, where the specific smoothness threshold
varies according to the distribution of primes. Running our code on
various ranges suggests that we can expect around $O(\sqrt{m})$ such
power-smooth numbers between 25 and $m$. In Phase 2, we use one of
the order $r$ elements suggested by Theorem~\ref{thm:LiebeckSaxlPrim}
to show invariable generation for these leftover numbers.

Thus, for each $n$ left over from Phase 1, we have already stored
the largest prime-power factor. As these values of $n$ are power-smooth,
it is efficient to use trial division to find two more large prime-power
factors. We keep the largest two odd prime-power divisors.

We then look for integers of the form $cr$ that are close to $n$,
where $r$ is prime and $c$ is small. Let $p^{a}$ be the largest
prime-power factor of $n$. As there is somewhat less to go wrong
with transitivity for smaller values of $c$, we start by looking
at $c$ from $2$ to $p^{a}/2$, and checking for a prime $r$ so
that $n-p^{a}<cr<n-2$. If smaller values of $c$ do not yield the
desired, then we factor each number between $n-p^{a}$ and $n-2$,
and take $r$ to be the prime factor (if any) that is larger than
$2\sqrt{n}$. 

We check for each candidate $cr$ that $n$ and $n-cr$ do not have
the form $\frac{q^{d}-1}{q-1}$ and $\frac{q^{i}-1}{q-1}$ for some
common $q$ (avoiding the primitive groups of Theorem~\ref{thm:LiebeckSaxlPrim}
not already eliminated by requiring $k>2$). This yields an element
of order $r$ which avoids all the primitive subgroups of $A_{n}$,
and which does not obviously fail to invariably generate $A_{n}$
with a prime-power element.

We now check primitivity for the subgroup generated by an $r$-element
and an element of order one of the two largest prime power divisors
of $n$. We apply Lemmas~\ref{lem:BasepDivis}, \ref{lem:CarryConditionImprim},
and \ref{lem:DivisImprim}. Specifically, we check for each $d$ dividing
$\gcd(c,n-c\cdot r)$ that multiplication of $d$ and $n/d$ yields
a carry in the associated prime base. Finally, we check for a system
of intransitivity over each partition of $n$ having a part $a\cdot r+b$,
where $a\leq c$ and $b\leq n-cr$.

Although the checks that we need to do for these numbers are considerably
more expensive than those of Phase 1, we only need to perform them
on a small collection of numbers. 
\begin{rem}
We prototyped the code for Phase 1 in GAP for ease of coding and experimentation.
We then ported the code from this phase to C, speeding the computation
by a factor of around 20.  The code from Phase 2 is not a speed bottleneck,
and the GAP library is convenient to use for some of the checks here,
so we have kept it in GAP.
\end{rem}

\begin{rem}
The average value of $c$ over large ranges has been between $2$
and $3$ in experiments. Occasional numbers from Phase 2 require a
$c$ that is larger, however. See also Section~\ref{subsec:BadN}.
\end{rem}

\subsection{Complexity}

Phase 1 is at its core a Segmented Sieve of Eratosthenes, requiring
$O(m\log\log m)$ operations and $O(\sqrt{m})$ space. Our use of
prime powers instead of primes makes no difference in order of complexity,
as prime powers have nearly the same density as primes. 

We also store the power-smooth numbers failing Phase 1: these also
appear experimentally to take just over $O(\sqrt{m})$ space. Indeed,
if we assume Cramer's conjecture \cite{Cramer:1936} (as certainly
holds in the ranges in which we are likely to compute), then prime
gaps are of at most $O(\log^{2}x)$ size. An estimate of Rankin (in
\cite{Rankin:1938}, see also \cite{Granville:2008}) gives the number
of $\log^{2}x$-smooth numbers in $[1,x]$ to be $x^{1/2+O(1/\log\log x)}$.
Assuming that the number of $\log^{2}x$-smooth numbers on an interval
on a length of $\log^{2}x$ is typically well-behaved, a estimate
for the number of failures is $O(\int_{1}^{m}\frac{\log^{2}x}{x}\cdot\sqrt{x}\,dx)=O(\sqrt{m}\cdot\log^{2}m)$.
Moderate optimizations in the use of space are likely possible, but
as time is a much bigger bottleneck, we have not pursued this. 

Phase 2 does not require any significant additional memory. It also
does not appear to take much additional time, although a careful time
analysis is more elusive. Let $\ell$ be the number of integers failing
Phase 1. Then for each, we must in the worst case find and check $O(\sqrt{\ell})$
primes $r$. The checks for a $PSL$ action and for primitivity are
not expensive. The check for transitivity on an integer $n$ of the
form $c\cdot r+k$ requires examining each number of the form $a\cdot r+b$,
where $0\leq a<c/2$ and $0\leq b\leq k$.

Although making the analysis careful is difficult here, we give a
heuristic argument for the time. By the same argument as in \cite[Section 5]{Shareshian/Woodroofe:2018},
we expect there to be a prime in the short interval $[(n-p^{a})/2,(n-3)/2]$
with high asymptotic density. We must also satisfy other transitivity
and primitivity tests, but these are also passed with high asymptotic
density. Thus, we may assume that $c=2$, so long as occasional numbers
requiring a higher associated $c$ do not cost greatly more time.
Assuming Cramer's conjecture, we also have $k<\log^{2}n$. Making
broad assumptions as above, we have $\ell=O(\sqrt{m}\cdot\log^{2}m)$.
Assuming that we may find primes quickly, the time is easily subsumed
within the $O(m\log\log m)$ time for Phase~1.

This heuristic bears up in practice, where Phase~2 runs quickly.
When both phases were implemented in GAP, Phase~2 takes only as much
time as that for a few segments in Phase~1.

\subsection{Small integers}

A few small cases included in Theorem~\ref{thm:MainTheorem} must
be handled separately. Theorem~\ref{thm:LiebeckSaxlPrim} has an
exception at $n=24$, but this meets the condition of Phase 1 (with,
say, $r=19$ and $p^{a}=2^{3}$). On the other hand, the numbers $n=6$
and $n=12$ cannot be handled by the checks in Phase 1. It is easy
to verify using GAP or facts about primitive subgroups that $A_{6}$
is generated by a $5$-cycle and a base-$2$ element, while $A_{12}$
is generated by an $11$-cycle and a base-$3$ element.

\begin{table}
\begin{centering}
\begin{tabular}{llll}
\toprule 
$p$ &
\multicolumn{3}{l}{base-$p$ representation (place order is from least to most significant)}\tabularnewline
\midrule
\addlinespace[0.3em]
$83$ &
$n$ &
= &
$0,30,72,44,52,50\quad=30\cdot p+72\cdot p^{2}+44\cdot p^{3}+52\cdot p^{4}+50\cdot p^{5}$\tabularnewline
 &
$2rp$ &
 &
$0,15,49,19,34,23$\tabularnewline
 &
$n-2rp$ &
 &
$0,15,23,25,18,27$\tabularnewline
\addlinespace[0.5em]
$53$ &
$n$ &
 &
$0,43,23,38,48,52,8$\tabularnewline
 &
$4rp$ &
 &
$0,37,11,34,33,16,5$\tabularnewline
 &
$n-4rp$ &
 &
$0,6,12,4,15,36,3$\tabularnewline
\addlinespace[0.5em]
$47$ &
$n$ &
 &
$0,32,38,30,29,23,18$\tabularnewline
 &
$3rp$ &
 &
$0,23,13,2,26,12,7$\tabularnewline
 &
$n-3rp$ &
 &
$0,9,25,28,3,11,11$\tabularnewline
\addlinespace[0.5em]
$29$ &
$n$ &
 &
$0,12,13,2,22,8,16,11$\tabularnewline
 &
 &
 &
transitive\tabularnewline
\addlinespace[0.5em]
$11$ &
$n$ &
 &
$0,10,6,5,8,8,7,4,6,7,7$\tabularnewline
 &
$6rp$ &
 &
$0,2,2,2,4,5,6,0,6,4,1$\tabularnewline
 &
$n-6rp$ &
 &
$0,8,4,3,4,3,1,4,0,3,6$\tabularnewline
\addlinespace[0.5em]
$7$ &
$n$ &
 &
$0,6,6,1,5,6,5,0,3,0,6,2,0,2$\tabularnewline
 &
 &
 &
transitive\tabularnewline
\addlinespace[0.5em]
$3$ &
$n$ &
 &
$0,0,0,2,2,1,1,2,0,1,2,2,0,0,2,0,1,2,1,0,0,1,0,2$\tabularnewline
 &
 &
 &
transitive\tabularnewline
\bottomrule
\end{tabular}\bigskip{}
\par\end{centering}
\caption{\label{tab:BadN}Prime divisors $p$ of $n=199\protect\comma445\protect\comma521\protect\comma968$,
with a system of intransitivity for a base-$p$ element and a base-$r$
element, if any. Here $r=555\protect\comma558\protect\comma557$.}
\end{table}

\subsection{\label{subsec:BadN}The integer $199\protect\comma445\protect\comma521\protect\comma968$}

Our computational strategy works without trouble for all integers
up to $10^{15}$ with five exceptions, as follow:\nopagebreak\smallskip{}

\begin{tabular}{lll}
$n_{1}=199\comma445\comma521\comma968$, &
$n_{2}=5\comma760\comma706\comma652\comma536$, &
$n_{3}=6\comma421\comma990\comma708\comma848$,\tabularnewline
$n_{4}=22\comma062\comma987\comma063\comma208$,~~ &
$n_{5}=138\comma057\comma417\comma511\comma650$.~~ &
\tabularnewline
\end{tabular}\smallskip{}

\noindent For these five integers, we modify the Phase~2 strategy
slightly to also examine smaller odd prime-power divisors. All but
$n_{1}$ and $n_{3}$ are generated by an element of large prime $r$
order (as in Phase~2), together with the base-$p$ element for $p$
the third largest (odd) prime-power divisor. For $n_{1}$ and $n_{3}$,
we must use the fourth largest prime-power divisor for $p$.

To give some insight, we examine carefully the situation with $n=n_{1}=199\comma445\comma521\comma968=2^{4}\cdot3^{3}\cdot7\cdot11\cdot29\cdot47\cdot53\cdot83$.
This value of $n$ has several candidate large primes $r$ that divide
a slightly smaller integer. We focus on $n-5=359\cdot555\comma558\comma557$,
so $r=555\comma558\comma557$. The corresponding $r$-element invariably
generates $A_{n}$ together with a base-$p$ element for the odd primes
$p=3,7,$ or $29$, but it fails transitivity for the primes $p=11,47,53,$
and $83$. In Table~\ref{tab:BadN}, we show the base-$p$ representation
for $n$ for each odd prime divisor $n$, together with a partition
of $n$ giving a system of intransitivity (if any).

Our main computation verified that the primes $p=83$ and $53$ fail
to give invariable generation with any candidate value of $r$. We
have verified by similar additional computation that the prime $p=47$
fails transitivity for with any candidate value of $r$. However,
the primes $p=3$ and $p=29$ each yield invariable generation of
$A_{n}$ at several values of $r\geq\sqrt{2n}$ (including $r=555\comma558\comma557$,
as we have already stated).

It is additionally worth noticing that $r=728\comma209>\sqrt{2n}$,
a divisor of $n-3$, fails to yield transitivity with any prime divisor
of $n$.

\section{\label{sec:Discussion}Discussion}

In light of our computational result, it is reasonable to ask:
\begin{question}
\label{que:MainQues}For every $n\geq5$, is the alternating group
$A_{n}$ invariably generated by an element whose order is a prime
power divisor $p^{a}$ of $n$, together with an element of prime
order $r>\sqrt{n}$?
\end{question}

We believe that counterexamples, if any, to the condition of Question~\ref{que:MainQues}
must be vanishingly rare. Theorem~\ref{thm:MainTheorem} says that
any counterexample must satisfy $n>10^{15}$. 

It would already be somewhat interesting to give a proof that does
not rely on the Riemann Hypothesis that the set of counterexamples
to Question~\ref{que:MainQues} has asymptotic density $0$. (The
second two authors gave a conditional proof of such a result with
the stronger restriction that $a=1$ in \cite[Section 5]{Shareshian/Woodroofe:2018}.)
\begin{note}
After this paper was arXived, Teräväinen answered the question implicit
in the previous paragraph in a strong form in \cite{Teravainen:2022UNP},
where he gave an explicit upper bound on the density of the set of
$n$ for which $A_{n}$ is not generated by two elements of prime
order.
\end{note}

\subsection{Number theoretic ingredients}

The strategy discussed in Section~\ref{subsec:StrategyGeneraton}
suggests that an important ingredient for answering Question~\ref{que:MainQues}
will be the existence of integers with large prime factors on short
intervals. This problem has been studied a certain amount in the number
theory literature. Most of the papers on this problem focus on intervals
of length close to that of $[x-\sqrt{x},x]$. For sufficiently large
$x$, a result of Harman \cite[Theorem 6.1]{Harman:2007} shows we
may find a number with a prime factor greater than $x^{37/50}$, while
one of Jia and Liu \cite{Jia/Liu:2000} yields a number with a prime
factor greater than $x^{25/26-\epsilon}$ on a slightly longer interval.

However, intervals of length $\sqrt{x}$ are already longer than we
generally require. The set of integers that are $\log^{1+\epsilon}x$-smooth
already have asymptotic density $0$, as shown by Rankin in \cite{Rankin:1938}.
Meanwhile, a prime factor of order $x^{1/2+\epsilon}$ will suffice
for use in Theorem~\ref{thm:LiebeckSaxlPrim}, and $x^{25/26}$ is
perhaps overkill. 

Recent results of Teräväinen \cite{Teravainen:2016} (improving on
earlier work by several authors) show that almost all intervals of
the form $[x,x+\log^{3.51}x]$ contain a number that is the product
of exactly two primes. Indeed, the more technical result \cite[Theorem 5]{Teravainen:2016}
shows that one of these primes must be at least $x/\log^{3}x$. Slightly
smaller primes (though still larger than $\sqrt{2x}$) on somewhat
shorter intervals may be produced from Theorem 4 in the same paper.
More recent related work of Matomäki \cite{Matomaki:2022} shows that
almost all intervals of length $\log^{1+\epsilon}x$ have either a
prime or a product of two primes.

\subsection{A heuristic for transitivity}

Since integers $n$ that are $\log^{1+\epsilon}n$-smooth have asymptotic
density $0$, we can with asymptotic density $1$ find $r$ that is
at least about $\sqrt{n}$ with $\lfloor\frac{n-3}{r}\rfloor\cdot r+p^{a}>n$,
where $p^{a}$ is the largest prime-power divisor of $n$. Indeed,
the $k=2$ and $k=1$ cases that we avoid discussing in Theorem~\ref{thm:LiebeckSaxlPrim}
are handled in \cite{Liebeck/Saxl:1985b}, and can occur only with
asymptotic density $0$, so may be ignored for the purpose of asymptotic
density arguments. 

We assume we can find such a prime $r$. Frequently, we will have
$\lfloor\frac{n-3}{r}\rfloor=1$, and then the desired invariable
generation holds by Lemma~\ref{lem:GoodSieve}. 

Otherwise, the main thing that can go wrong is for $n$ to admit an
integer partition into two parts $n_{1}+n_{2}$ so that $n_{1}=a\cdot r+b$
for $b<p^{a}$, and so that the base-$p$ representation of $n_{1}$
has every digit smaller than the corresponding digit of the base-$p$
representation of $n$. As seen in Example~\ref{exa:31416} and in
Section~\ref{subsec:BadN}, this indeed can happen for certain $p$.
While we do not have a sketch that we know how to make precise, we
know that the condition of Lemma~\ref{lem:GoodSieve} will fail only
for highly power-smooth $n$. In this situation, the integer $n$
will be divisible by several prime-powers of about the same size.
Heuristically, changing the prime slightly will tend to alter the
digits in the base-$p$ expression greatly.

The example of Section~\ref{subsec:BadN} shows the extent to which
this heuristic may fail.

\bibliographystyle{1_Users_russw_Documents_Research_mypapers_Invariable_generation_of_alternating_groups_hamsplain}
\bibliography{0_Users_russw_Documents_Research_mypapers_Invariable_generation_of_alternating_groups_Master}

\end{document}